\documentclass[reqno]{amsart}

\usepackage{hyperref} 
\usepackage{graphicx, float}
\usepackage{amsmath, amsthm, amsfonts, amssymb, physics}
\usepackage{tikz}

\def\ex{{\rm ex}}

\newcommand{\floor}[1]{\left\lfloor #1 \right\rfloor}

\newcommand\degvec[1]{\langle #1 \rangle}
\newcommand\numberthis{\addtocounter{equation}{1}\tag{\theequation}}

\newtheorem{theorem}{Theorem}[section]
\newtheorem{lemma}[theorem]{Lemma}
\newtheorem{proposition}[theorem]{Proposition}
\newtheorem{corollary}[theorem]{Corollary}
\newtheorem{remark}[theorem]{Remark}
\newtheorem{conjecture}[theorem]{Conjecture}

\theoremstyle{definition}

\newtheorem{definition}[theorem]{Definition}

\title{Crowns in linear $3$-graphs}

\date{\today}

\author[A. Carbonero]{Alvaro Carbonero}
\address{Department of Combinatorics and Optimization, University of Waterloo, ON, Canada}
\email{ar2carbo@uwaterloo.com}

\author[W. Fletcher]{Willem Fletcher}
\address{Department of Mathematics and Statistics, Carleton College, Northfield, MN, USA}
\email{willemrfletcher@gmail.com}

\author[J. Guo]{Jing Guo}
\address{Department of Mathematics, University of Utah, Salt Lake City, UT, USA}
\email{math.guoj@gmail.com}

\author[A. Gy\'arf\'as]{Andr\'as Gy\'arf\'as}
\address{Alfr\'ed R\'enyi Institute of Mathematics, Budapest, Hungary}
\email{gyarfas.andras@renyi.hu}

\author[R. Wang]{Rona Wang}
\address{\parbox{\linewidth}{Department of Mathematics, Massachusetts Institute of Technology, Cambridge, \\ MA, USA}}
\email{rona@mit.edu}

\author[S. Yan]{Shiyu Yan}
\address{Department of Mathematics and Statistics, Carleton College, Northfield, MN, USA}
\email{math.shiyu.yan@gmail.com}

\thanks{The work presented here was done as part of the Budapest Semesters in Mathematics Summer Undergraduate Research Program in the summer of 2021 under the supervision of the fourth author. } 

\begin{document}
    
    \begin{abstract}
        A \textit{linear $3$-graph}, $H = (V, E)$, is a set, $V$, of vertices together with a set, $E$, of $3$-element subsets of $V$, called edges, so that any two distinct edges intersect in at most one vertex. The linear Tur\'an number, $\ex(n,F)$, is the maximum number of edges in a linear $3$-graph $H$ with $n$ vertices containing no copy of $F$.
        
        We focus here on the \textit{crown}, $C$, which consists of three pairwise disjoint edges (jewels) and a fourth edge (base) which intersects all of the jewels. Our main result is that every linear $3$-graph with minimum degree at least $4$ contains a crown. This is not true if $4$ is replaced by $3$. In fact the known bounds of the Tur\'an number are
        \[ 6 \floor{\frac{n - 3}{4}} \leq \ex(n, C) \leq 2n, \]
        and in the construction providing the lower bound all but three vertices have degree $3$. We conjecture that $\ex(n, C) \sim \frac{3n}{2}$ but even if this were known it would not imply our main result.
        
        Our second result is a step towards a possible proof of $\ex(n,C) \leq \frac{3n}{2}$ (i.e., determining it within a constant error). We show that a minimal counterexample to this statement must contain certain configurations with $9$ edges and we conjecture that all of them lead to contradiction.
    \end{abstract}
    
    \maketitle
    
    \section{Introduction}
    
    A \textit{$3$-graph}, $H = (V,E)$, is a set, $V$, whose elements are called points or vertices together with a set, $E$, of $3$-element subsets of $V$ called edges. If not clear from the context, we use the notation $V(H)$ and $E(H)$ for $V$ and $E$ respectively. We restrict ourselves to the important family of \textit{linear} $3$-graphs where \textit{any two distinct edges intersect in at most one vertex}. In the remainder of this paper we use the term $3$-graph for linear $3$-graph.
    
    The number of edges containing a point $v \in V(H)$ is the \textit{degree} of $v$ and is denoted by $d(v)$ or $d_H(v)$. We denote by $\delta (H)$ the minimum degree of $H$. Similar notations are used for graphs ($2$-uniform linear hypergraphs). We use $[k]$ to denote $\{1, \ldots, k\}$.
    
    Let $F$ be a fixed $3$-graph. A $3$-graph, $H$, is called \textit{$F$-free} if $H$ has no subgraph isomorphic to $F$. The \textit{(linear) Tur\'an number of $F$}, $\ex(n,F)$, is the maximum number of edges in an $F$-free $3$-graph on $n$ vertices. 
    
    The behavior of $\ex(n,F)$ is interesting even if $F$ has three or four edges. A famous theorem of Ruzsa and Szemer\'edi \cite{RSZ} is that $\ex(n,T) = o(n^2)$ if $T$ is the \textit{triangle}. For the \textit{Pasch configuration}, $P$, $\ex(n,P) = \frac{n (n-1)}{6}$ for infinitely many $n$ since there are $P$-free Steiner triple systems (see \cite{CR}). For the \textit{fan}, $F$, we have $\ex(n,F) = \frac{n^{2}}{9}$ if $n$ is divisible by $3$ (see \cite{FGY}). Figure \ref{fig-ex} shows these $3$-graphs (drawn with the convention that edges are represented as straight line segments).
    
    \begin{figure}[H]
        \centering
        \includegraphics[width=0.6\linewidth]{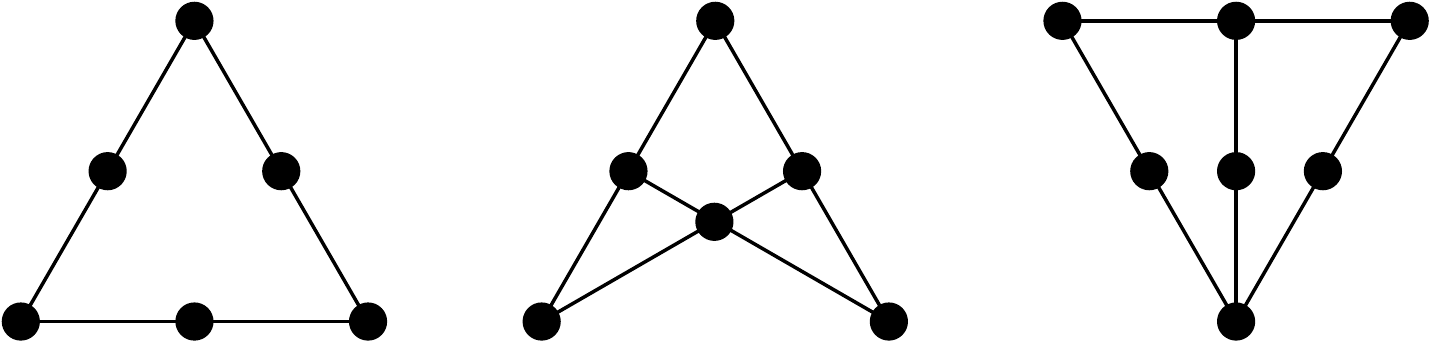}
        \caption{Triangle, Pasch configuration, and Fan}
        \label{fig-ex}
    \end{figure}
    
    This paper is related to the Tur\'an number of the \textit{crown}, $C$ (Figure \ref{fig-crown}).
    
    \begin{figure}[H]
        \centering
        \includegraphics[width=0.25\linewidth]{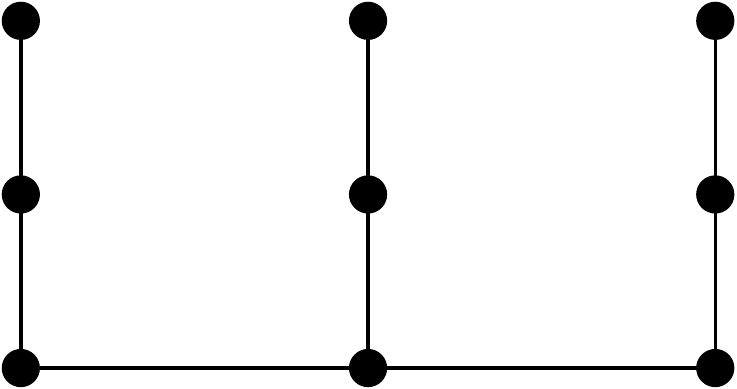}
        \caption{The crown}
        \label{fig-crown}
    \end{figure}
    
    It seems that $C$ deserves a descriptive name, the list of small configurations in \cite{CR} refers to it as $C_{13}$. We call the horizontal edge of the crown the \textit{base} and the vertical edges \textit{jewels}. The crown is the smallest $3$-tree with unknown Tur\'an number, the following bounds are from \cite{GYRS}:
    \begin{equation} \label{crownbound}
        6 \floor{\frac{n - 3}{4}} \leq \ex(n, C) \leq 2n.
    \end{equation}
    The construction for the lower bound in (\ref{crownbound}) (for the case $n \equiv 3 \pmod 4$) is the following.  Choose three vertices $\{ a, b, c \}$, and define  edges
    \[ (a,x_i,y_i), \ (a,z_i,w_i), \ (b,x_i,w_i), \ (b,y_i,z_i), \ (c,x_i,z_i), \ (c,y_i,w_i) \]
    where $i = 1, 2, \ldots, \floor{ (n - 3) / 4}$ and $x_i$, $y_i$, $z_i$, and $w_i$ are distinct vertices.
    
    In this construction (overlay of Fano planes) all but three vertices have degree $3$. This poses the question whether raising the minimum degree of a $3$-graph $H$ from $3$ to $4$ ensures a crown. Our main result is an affirmative answer.
    
    \begin{theorem} \label{main-theorem}
        Every $3$-graph with minimum degree $\delta(H) \geq 4$ contains a crown.
    \end{theorem}
    
    It is worth noting that even if $\ex(n, C) \leq \frac{3n}{2}$ were known, Theorem \ref{main-theorem} would not follow since minimum degree $4$ ensures only $\frac{4n}{3} < \frac{3n}{2}$ edges. 
    
    For $e = (a, b, c) \in E(H)$, let $D(e)$ denote the degree vector $\degvec{ d(a), d(b), d(c) }$ with coordinates in non-increasing order. We define a partial order on these vectors by considering $D(e) \geq D(f)$ if at all positions the coordinate of $e$ is larger than or equal to that of $f$.
    
    One tool in proving the upper bound $\ex(n,C) \leq 2n$ in \cite{GYRS} was showing that in a crown-free $3$-graph, the set of $11$ edges incident to an edge, $e$, with $D(e) = \degvec{5,5,3}$ form two possible $3$-graphs. To prove Theorem \ref{main-theorem} we need a similar result for the case $D(e) = \degvec{4,4,4}$. Lemma \ref{triple-four-lemma} in Section \ref{sectlinkgraph} proves that we have $5$ possible $3$-graphs in this case.
    
    Our second result is a ``reduction theorem'' showing that the (almost) sharp upper bound $\ex(n,C) \leq \frac{3n}{2}$ would follow if edges with $D(e) \geq \degvec{4,4,3}$ or $D(e) \geq \degvec{5,4,2}$ are not present. It is worth noting that $D(e) \geq \degvec{6,4,2}$ is not possible in a crown-free $3$-graph (an easy exercise).
    
    \begin{theorem} \label{sec-theorem}
       Assume that a crown-free $3$-graph, $H$, with $n$ vertices has no edge $e \in E(H)$ with $D(e) \geq \degvec{4,4,3}$ or $D(e) \geq \degvec{5,4,2}$. Then $\abs{E(H)} \leq \frac{3n}{2}$.
    \end{theorem}
    
    A \textit{critical configuration} in a crown-free $3$-graph is defined by the $9$ edges incident to an edge $e$ with $D(e) = \degvec{4,4,3}$ or with $D(e) = \degvec{5,4,2}$. An immediate corollary of Theorem \ref{sec-theorem} is the following:
    
    \begin{corollary} \label{cor}
        If a crown-free $3$-graph $H$ contains no critical configuration then $\abs{E(H)} \leq \frac{3n}{2}$.
    \end{corollary}
    
    Corollary \ref{cor} may lead to a proof of $\ex(n,C) \leq \frac{3n}{2}$, since minimal counterexamples (with $n$ as small as possible) probably cannot contain critical configurations. 
    
    \begin{conjecture} \label{conj}
        Minimal counterexamples to $\ex(n,C) \leq \frac{3n}{2}$ cannot contain critical configurations.
    \end{conjecture}
   
    We prove Conjecture \ref{conj} for one particular critical configuration in Section \ref{concluding-section}. The method seems to work for all others but new ideas are needed to achieve a reasonably short proof this way.
    
    In Sections \ref{sectlinkgraph} and \ref{secgoodquint} we define our tools. In Sections \ref{sectproof1} and \ref{sectproof2} we prove Theorems \ref{main-theorem} and \ref{sec-theorem}.

    \section{Link graphs of edges with $D(e) = \degvec{4, 4, 4}$}
    \label{sectlinkgraph}
    
     \begin{definition}[Link graph of an edge] \label{linkgraph}
         Assume that $H$ is a $3$-graph and $e = (a,b,c) \in E(H)$. The \textit{link graph}, $G(e)$, is the graph whose edges are the pairs $(x,y)$ for which there exists $(x,y,z) \in E(H)$ with $z \in \{a,b,c\}$. The set of vertices of $G(e)$ is defined as the subset of $V(H)$ covered by the edges of $G(e)$. 
     \end{definition}
    
    Note that Definition \ref{linkgraph} provides a proper $3$-coloring of the edges of $G(e)$ with colors $a$, $b$, and $c$. We denote by $\varphi(x,y)$ the color of the edge $(x,y)$ in this coloring. Edges with colors $a$, $b$, and $c$ will be labelled $\alpha$, $\beta$, and $\gamma$, respectively and are colored red, blue, and green in colored figures. Observe that a crown with base edge $e$ exists in $H$ if and only if $G(e)$ has three pairwise disjoint edges with different colors,  which we call a \textit{rainbow matching}.
    
    \begin{lemma} \label{triple-four-lemma}
        If a crown-free $3$-graph $H$ has an edge $e = (a, b, c)$ such that $D(e) = \degvec{4, 4, 4}$, then $G(e)$ is isomorphic (up to permutation of colors) to one of the following five graphs (see Figure \ref{fig-link-graphs}).
    \end{lemma}
    
    \begin{figure}[H]
        \centering
        \includegraphics[width=0.75\linewidth]{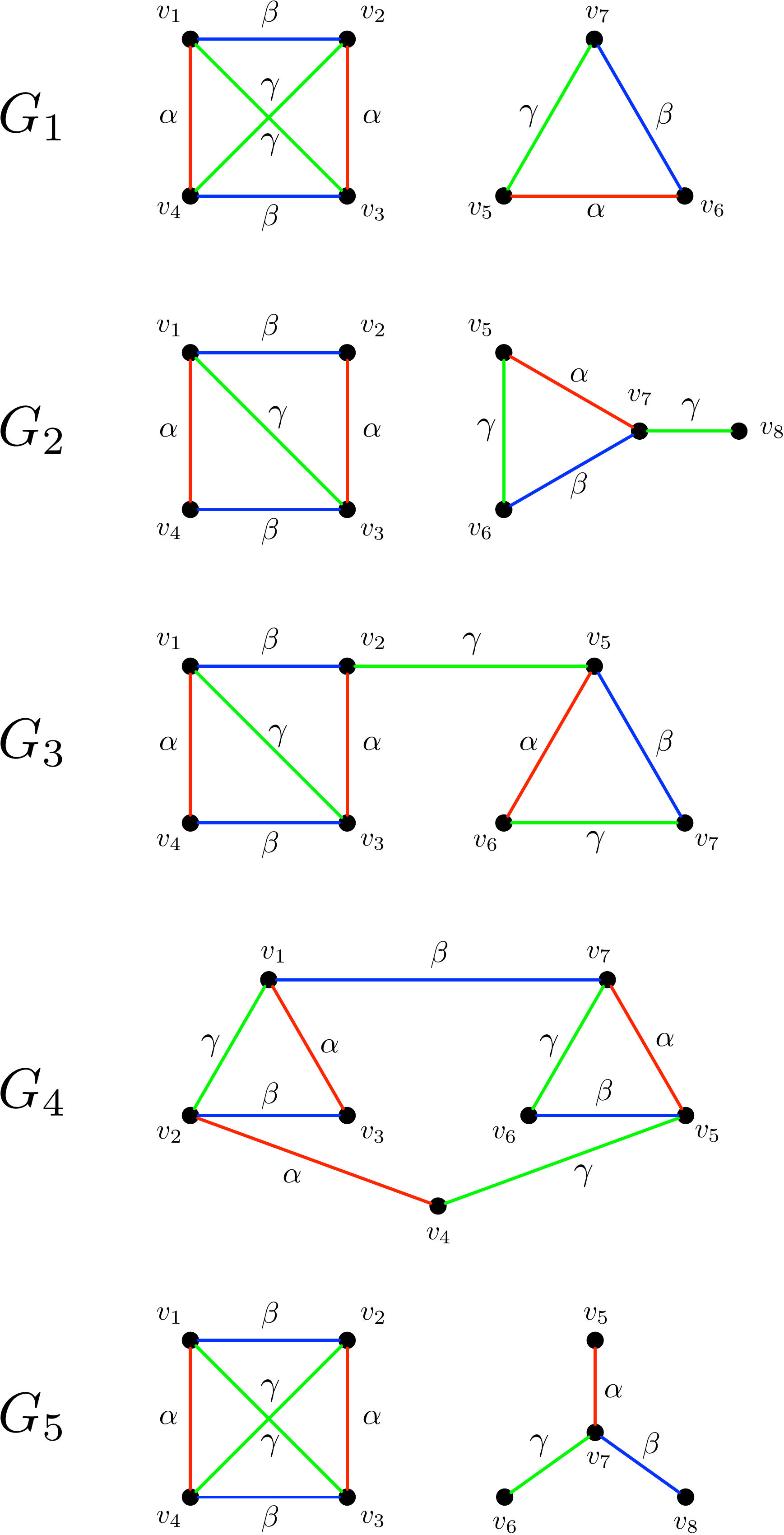}
        \caption{Link graphs $G_1, \ldots, G_5$}
        \label{fig-link-graphs}
    \end{figure}
    
    \begin{proof}
        For $i \in [3]$, let $M_i$ denote the vertex set of the matching of color $a$, $b$, and $c$ in $G(e)$, respectively. Observe that for all $i, j \in [3]$, where $i \neq j$, $M_i$ must intersect all the three edges in $M_j$. Otherwise, there exists an edge $f \in M_j$ not intersecting $M_i$ and an edge $g \in M_k$, for $k \neq i \text{ or } j$, such that $g \cap f = \emptyset$. Then $f$, $g$, and some edge in $M_i$ is a rainbow matching.

        First, we show that $\abs{M_1 \cap M_2} > 3$. It follows from the previous observation that $\abs{M_1 \cap M_2} \geq 3$. Assume for contradiction that $\abs{M_1 \cap M_2} = 3$. Then $S = M_1 \cap M_2$ cannot contain an edge from $M_1 \cup M_2$, otherwise we get a contradiction with our observation. Thus, $M_1$ and $M_2$ are matchings from $S$ to $M_1 \setminus S$ and from $S$ to $M_2 \setminus S$, respectively. If $e \in M_3$, then $e \in S$ since otherwise $e$ would be part of a rainbow matching. However, it is impossible for there to be three disjoint edges in $S$ since $\abs{S} = 3$ by assumption.
        
        Now, suppose $\abs{M_1 \cap M_2} = 4$. Then, in $S$ there is one edge
        of $M_1$ and one edge of $M_2$. If the two edges are disjoint, we have two disjoint $\alpha$-$\beta$ paths both with four vertices. One has two $\alpha$-edges and one $\beta$-edge, and the other has the opposite. To avoid a rainbow matching, any $\gamma$-edge must intersect every edge in one of the paths. There are thus four possible locations for a $\gamma$-edge. Choose any three of them gives $G_2$. On the other hand, if the two edges intersect, we will get a contradiction. In this case, the edges of $M_1 \cup M_2$ form two disjoint, alternating $\alpha$-$\beta$ paths with three and five vertices, respectively. Following the paths, label the three vertices $v_1$ through $v_3$, and the five vertices $w_1$ through $w_5$. The only possible location for a $\gamma$-edge that does not intersect $v_2$ and would not form a rainbow matching is $(w_2, w_4)$. However, then at least two $\gamma$-edges must intersect $v_2$, a contradiction.
        
        We show that the next case implies that $G(e)$ is isomorphic (up to permutation of colors) to one of the $G_i$'s.
        
        Assume $\abs{M_1 \cap M_2} = 5$. Then either $M_1 \cup M_2$ is an alternating $\alpha$-$\beta$ path on seven vertices, or it is a disjoint alternating $\alpha$-$\beta$ four-cycle and $\alpha$-$\beta$ path on three vertices. 
        
        In the former case, label the vertices $v_1$ through $v_7$ along the path. Then the possible $\gamma$-edges that don't create a rainbow matching are $(v_1, v_3)$, $(v_1, v_6)$, $(v_2, v_4)$, $(v_2, v_6)$, $(v_2, v_7)$, $(v_4, v_6)$, and $(v_5, v_7)$. Apart from the symmetry (reflection of a point of the path through $v_4$), the non-intersecting triples of these edges are
        \[ \{ (v_1, v_3), (v_2, v_4), (v_5,v_7) \}, \]
        \[ \{ (v_1, v_3), (v_2, v_7), (v_4,v_6) \}, \ \{ (v_1, v_3), (v_2, v_6), (v_5,v_7)\}. \]
        The first triple gives $G_3$, and the last two triples give $G_4$. 
        
        In the latter case, label the vertices $v_1$ through $v_3$ along the path and $w_1$ through $w_4$ along the cycle. If there is a $\gamma$-edge containing a vertex, $u \notin M_1 \cup M_2$, then that edge must be $(u, v_2)$ or part of a rainbow matching. Thus, there is at most one such edge. The vertex $v_2$ may also be in a $\gamma$-edge with $w_i$ for $i \in [4]$; however, it may only be in one so we only consider the edge $(v_2, w_1)$. The other possible $\gamma$-edges are $(v_1, v_3)$, $(w_1, w_3)$, and $(w_2, w_4)$. Thus, apart from the symmetry (choice of diagonal in the cycle), the possible non-intersecting triples from these edges are
        \[ \{ (v_1, v_3), (w_1, w_3), (w_2, w_4) \}, \ \{ (v_1, v_3), (u, v_2), (w_1, w_3) \}, \]
        \[ \{ (v_1, v_3), (v_2, w_1), (w_2, w_4) \}, \ \{ (u, v_2), (w_1, w_3), (w_2, w_4) \}. \]
        These triples give $G_1$, $G_2$, $G_3$, and $G_5$, respectively.
        
        Lastly, suppose $\abs{M_1 \cap M_2} = 6$. Then $M_1 \cup M_2$ is an alternating $\alpha$-$\beta$ six-cycle. Any $\gamma$-edges
        intersecting the cycle in at most one vertex or along a long diagonal are in rainbow matchings. However, at most two $\gamma$-edges can be short diagonals without intersecting, which is a contradiction thus concludes the proof.
    \end{proof}

    \section{Good quintuple lemma}
    \label{secgoodquint}
    
    As shown in Section \ref{sectlinkgraph}, it is easy to recognize a crown with base edge $e$: We have to find a rainbow matching in $G(e)$. To recognize other crowns related to $G(e)$, we introduce the following definition:
    
    \begin{definition}[Good quintuple]
        A quintuple $Q=\{ x_{1}, x_{2}, x_{3}, x_{4}, x_{5} \}$ of vertices of $G(e)$ is \textit{good} if
        \begin{itemize}
            \item $x_{1} x_{2}$, $x_{2} x_{3}$, and $x_{4} x_{5}$ are edges of $G(e)$
            \item $\varphi(x_{1} x_{2}) = \varphi(x_{4} x_{5})$
        \end{itemize}
    \end{definition}
    
    \begin{remark}\label{quintremark}
        The ordering of the vertices in $Q=\{ x_1,x_2,x_3,x_4,x_5 \}$ is important. Assume that $Q$ is a good quintuple. Then the quintuple $\{ x_1,x_2,x_3,x_5,x_4 \}$ is still good. However, observe that $\{ x_2,x_1,x_3,x_4,x_5 \}$ is good if and only if $(x_1,x_3)$ is an edge in $G(e)$. On the other hand, $\{ x_3,x_2,x_1,x_4,x_5 \}$ is never good.
    \end{remark}
    
    \begin{remark} \label{allbutv7}
        Observe (see Figure \ref{fig-link-graphs}) that apart from $v_7 \in V(G_5)$, every vertex in each $G_i$ is the first vertex of some good quintuple.
    \end{remark}
    
    \begin{lemma}[Good quintuple lemma] \label{quintuple-lemma}
        Assume $H$ is a crown-free $3$-graph and $Q = \{ x_{1}, x_{2}, x_{3}, x_{4}, x_{5} \}$ is a good quintuple in $G(e)$ for some $e = (a, b, c) \in E(H)$. Then there is no edge $f \in E(H)$ such that $f \cap e = \emptyset$ and that $f \cap Q = \{x_1\}$.
    \end{lemma}
    
    \begin{proof}
        Without loss of generality, $Q$ defines the edges $(x_1,x_2,a)$, $(x_2,x_3,b)$, and $(x_4,x_5,a)$ in $H$. Assume towards contradiction that edge $f = (p,q,x_1)$ where (from the assumptions) $p, q \notin Q \cup \{ a,b,c \}$.
        
        Observe that
        \[(p,q,x_1), (x_2,x_3,b), (x_4,x_5,a)\]
        are pairwise disjoint edges and $(x_1,x_2,a)$ intersects all of them, thus we have a crown (with base $(x_1,x_2,a)$), a contradiction.
    \end{proof}

    \section{Proof of Theorem \ref{main-theorem}}
    \label{sectproof1}
        
        Suppose that Theorem \ref{main-theorem} is not true, there exists a crown-free $3$-graph $H$ with $\delta(H) \geq 4$. Select an arbitrary edge $e = (a,b,c) \in E(H)$ and let $H'$ be the $3$-graph obtained from $H$ by removing edges intersecting $e$ until $D(e) = \degvec{4,4,4}$ in $H'$. Then Lemma \ref{triple-four-lemma} can be applied to $H'$ and we get that $G_i \subseteq G(e)$ for some $i \in [5]$. Further, note that every vertex $v$ in $G_i$ has degree at most three in $H'$, thus we can select $f_v \in E(H)$ such that $v \in f_v$ and $f_v \cap e = \emptyset$. Selecting $v \neq v_7$, there exists a good quintuple $Q$ with first vertex $v$ in $G_i$ (see Remark \ref{allbutv7}). We shall get a contradiction from Lemma \ref{quintuple-lemma}, finding a good quintuple $Q$ satisfying $f_v \cap Q = \{v\}$. This is obvious if $f_v \cap V(G_i) = \{v\}$, therefore in the subsequent cases we may assume that $f_v = (v,p,q)$ where $v,p \in V(G_i)$.

        \begin{itemize}
        \item $G(e) = G_1$. Set $v=v_1$ and from the symmetry of $G_1$ we may assume that $f_{v_1} = (v_1, v_5, q)$ (where $q \notin V(G_1)$).  Then $Q = \{v_1,v_2,v_3,v_6,v_7\}$ is a good quintuple. 
        
        \item $G(e) = G_2$. Set $v=v_1$ and (apart from symmetry) we have to consider either $f_{v_1} = (v_1,v_5,q)$ (where $q=v_8$ is possible) or $f_{v_1} = (v_1,p,q)$ where $p \in \{v_7,v_8\}$ and $q \notin V(G_2)$. In the former case $Q = \{v_1,v_2,v_3,v_6,v_7\}$ and in the latter $Q = \{v_1,v_3,v_4,v_5,v_6\}$ is a good quintuple.
        
        \item $G(e) = G_3$. Set $v=v_2$ and (up to symmetry) we have to consider either $f_{v_2} = (v_2, v_6, q)$ (where $q = v_4$ is possible) or $f_{v_2} = (v_2, v_4, q)$ (where $q\not \in G(e)$). In both cases $Q = \{v_2, v_5, v_7, v_1, v_3 \}$ is a good quintuple.
        
        \item $G(e) = G_4$. Set $v=v_1$. We have to consider three cases: either $f_{v_1} = (v_1, v_4, q)$ (where $q = v_6$ is possible), $f_{v_1} = (v_1, v_6, q)$ (where $q \not = v_4$), and $f_{v_1} = (v_1, v_5, q)$. In the first two cases $Q = \{v_1, v_3, v_2, v_5, v_7 \}$ is a good quintuple, and in the last case $ Q = \{v_1, v_2, v_3, v_6, v_7 \}$ is a good quintuple.
        
        
       
        \item $G(e) = G_5$. Set $v=v_5$. Assume first that $f = f_{v_5} = \{v_5, v_6, v_8\}$. In this case $d_H(a) = d_H(b) = d_H(c) = 4$, i.e. $H=H'$ since otherwise we have an edge $g$ intersecting $e$ and intersecting $V(G_5)$ in at most one point (in one of $v_5,v_6,v_8$). Then $g$ would be a jewel in a crown with base $e = (a,b,c)$ leading to contradiction. Since $d_{H'}(v_5) \geq 3$, there exists $f' = f'_{v_5} = (v_5,p,q) \in E(H)$ different from $f_{v_5}$ and from $(a,v_5,v_7)$. Since $f' \cap e = \emptyset$ (from $H=H'$), we can select $f'$ instead of $f$. Up to symmetry, we may assume that $p = v_6$ and $q = v_1$. Then $Q = \{v_5, v_7, v_8, v_2, v_3 \}$ is a good quintuple.
       \end{itemize}
       Since all cases ended by finding $v \in V(G_i)$, $f_v \in E(H)$, and a good quintuple $Q$ such that $f_v \cap e = \emptyset$ and $f_v \cap Q = \{v\}$, we get a contradiction from Lemma \ref{quintuple-lemma}, concluding the proof. \qed
        
       
    \section{Proof of Theorem \ref{sec-theorem}}
    \label{sectproof2}
        
        Suppose that Theorem \ref{sec-theorem} is not true: Let $H$ be a minimal counterexample, that is, a crown-free $3$-graph satisfying the conditions in Theorem \ref{sec-theorem} and $\abs{E(H)} > 3n/2$, with $n$ as small as possible. If there is a vertex $v \in V(H)$ with degree $d(v) \leq 1$, we remove $v$ together with the one possible edge containing $v$ and get a smaller counterexample, contradicting the minimality of $H$. Thus $d(v) \geq 2$ holds for all $v \in V(H)$.
        
        We define the partition $V(H) = Y \cup Z$ where $Z$ contains the vertices of degree at most three and $Y$ is the set of remaining vertices (of degree at least four).
        
        A \textit{special vertex} is a vertex $v$ with $d_{H}(v) = 2$, such that for the two edges $(a_1, a_2, v)$ and $(b_1, b_2, v)$ containing $v$ we have
        \[ d_{H}(a_1) = d_{H}(a_2) = d_{H}(b_1) = d_{H}(b_2) = 4. \]
        
        Partition $Z$ into three parts as follows: Let $Z_1$ be the set of vertices in $Z$ with $d_H(v) = 3$, $Z_2$ be the set of non-special vertices in $Z$ with $d_H(v) = 2$, and $Z_3$ be the set of special vertices. Let $E_1$ denote the set of edges in $H$ intersecting $Z$ in at least two vertices and set $E_2 = E(H) \setminus E_1$. Since there is no edge $(a, b, c)$ in $E(H)$ with $\degvec{a, b, c} \geq \degvec{4, 4, 4}$, all edges in $E_2$ intersect $Z$ in exactly one vertex. See Figure \ref{fig:prop51fig} illustrating the definitions, where the numbers indicate degrees and edges of $E_2$ are dotted.
        
        \begin{figure}[h!tpb]
    \centering
    \begin{tikzpicture}[block1/.style ={rectangle, draw=black, thin, text width=4em, minimum height=21em, text height=15em, align=center}, block2/.style ={rectangle, draw=black, thin, text width=4em, minimum height=21em, text height=15em, align=center}]
    
    \draw[color=black, thick] (1,3) -- (1,1) -- (-1,1) -- (-1, 3) -- cycle;
    \draw[color=black, thick] (-1,1) -- (-1, -0.5) -- (1, -0.5) -- (1, 1);
    \draw[color=black, thick] (-1, -0.5) -- (-1, -2) -- (1, -2) -- (1, -0.5);
    
    \node[] at (-1.5, 2)  (z1label) {$Z_1$};
    \node[] at (-1.5, 0.25)  (z2label) {$Z_2$};
    \node[] at (-1.5, -1.25)  (z3label) {$Z_3$};
    
    \draw[color=black, thick] (6,3) -- (6,1.5) -- (4,1.5) -- (4, 3) -- cycle;
    
    \draw[color=black, thick] (4, 1.5) -- (4, -2) -- (6, -2) -- (6, 1.5);

    \node[] at (6.5, 0.5)  (ylabel) {$Y$};
    \node[] at (5.5, -1.75) (y1label) {$Y_1$};
    
    \draw[color=black, thick] (-0.4, 2.5) -- (0.25, 0.55);
    \draw[color=black, thick] (0.5, 2.5) -- (0.573333, 0.3);
    \draw[color=black, thick] (-0.4, 2.5) -- (5, 2.5);
    \draw[color=black, thick] (-0.4, 2.5) -- (5.2, 0.653846);
    \draw[color=red, dashed] (0, -1.5) -- (5.2, 0.653846);
    \draw[color=red, dashed] (0.25, 0.55) -- (5.5, 0.320118177);

    \filldraw[black] (-0.4, 2.5) circle (2pt) node[inner sep=5pt, anchor=east] {$3$};
    \filldraw[black] (0.5, 2.5) circle (2pt) node[inner sep=5pt, anchor=south] {$3$};
    \filldraw[black] (-0.1, 1.6) circle (2pt) node[inner sep=3pt, anchor=north east] {$3$};
    \filldraw[black] (0.51, 2.2) circle (2pt) node[inner sep=5pt, anchor=north west] {$3$};
    
    \filldraw[black] (0.25, 0.55) circle (2pt) node[inner sep=5pt, anchor=east] {$2$};
    \filldraw[black] (0.573333, 0.3) circle(2pt) node[inner sep=5pt, anchor=north] {$2$};
    
    \filldraw[black] (0, -1.5) circle(2pt) node[inner sep=5pt, anchor=east] {$2$};
    
    \filldraw[black] (5, 2.5) circle(2pt) node[inner sep=5pt, anchor=west] {$5$};
    
    \filldraw[black] (5.2, 0.653846) circle(2pt) node[inner sep=5pt, anchor=west] {$4$};
    \filldraw[black] (4.5, 0.363905192) circle(2pt) node[inner sep=5pt, anchor=north] {$4$};
    \filldraw[black] (5.5, 0.320118177) circle(2pt) node[inner sep=5pt, anchor=west] {$4$};

    \end{tikzpicture}

    \caption{}
    \label{fig:prop51fig}
\end{figure}
        
        \begin{proposition} \label{prop}
            For all $v \in Z_1$, we have $d_{E_2}(v) = 0$. For all $v \in Z_2$, we have $d_{E_2}(v) \leq 1$. Moreover, $\abs{Z_3} \leq \abs{Y}$.
        \end{proposition}
        
        \begin{proof}
            Since $H$ has no edge $(a, b, c) \in E(H)$ with $\degvec{a, b, c} \geq \degvec{4, 4, 3}$, for all $v \in Z_1$ we have $d_{E_2}(v) = 0$, proving the first statement.
            
            Assume that for some vertex $v \in Z_2$, $d_{E_2}(v) = 2$. Since H has no edge with $\degvec{a, b, c} \geq \degvec{5, 4, 2}$, the two edges in $E_2$ containing $v$ have all other vertices of degree $4$. But then $v \in Z_3$, contradicting the assumption and proving the second statement.
            
            To prove the third statement, let $Y_1$ denote the set of vertices with degree $4$ in $Y$. Let $G$ be the graph with vertex set $Y_1$ and let $(y,y')$ be an edge of $G$ for $y,y' \in Y_1$ if and only if there exists $z \in Z_3$ such that $(y,y',z) \in E_2$. Then
            \[ 2 \abs{Z_3} = \abs{E(G)} = \frac{1}{2} \sum_{y \in Y_1} d_{G}(y) \leq \frac{1}{2} \sum_{y \in Y_1} d_{H}(y) = 2 \abs{Y_1} \leq 2 \abs{Y}, \]
            implying $\abs{Z_3} \leq \abs{Y}$ as required.
        \end{proof}
            
            
        
        Note that $\sum_{v \in Z} d_{E_1} (v)$ is at least a double-count of $\abs{E_1}$ from the definition of $E_1$. Therefore, using that $d_{E_1} (v) = 2 - d_{E_2} (v)$ for $v \in Z_2$ and that $d_{E_1} (v) = 0$ for $v \in Z_3$, we have
        \[ \abs{E_1} \leq \frac{1}{2} \left( \sum_{v \in Z_1} d_{E_1} (v) + \sum_{v \in Z_2} ( 2 - d_{E_2} (v) ) \right). \]
                From the first statement of Proposition \ref{prop}, we have
        \[ \abs{E_2} = \sum_{v \in Z_2} d_{E_2} (v) + \sum_{v \in Z_3} d_{E_2} (v). \]
                Therefore,
        \begin{align*}
        \abs{E(H)} & = \abs{E_{1}} + \abs{E_{2}} \\
        & \leq \frac{1}{2} \left( \sum_{v \in Z_1} d_{E_1}(v) + \sum_{v \in Z_2} (2 - d_{E_2}(v)) \right) + \sum_{v \in Z_2} d_{E_2}(v) + \sum_{v \in Z_3} d_{E_2}(v) \\
        & = \frac{1}{2} \left( \sum_{v \in Z_1} d_{E_1}(v) + 2 \abs{Z_2} + \sum_{v \in Z_2} d_{E_2}(v) \right) + \sum_{v \in Z_3} d_{E_2}(v) \\
        & \leq \frac{1}{2} \left( 3 \abs{Z_{1}} + 3 \abs{Z_{2}} \right) + 2 \abs{Z_{3}} \numberthis \label{first-ineq} \\
        & \leq \frac{1}{2} \left( 3 \abs{Z_{1}} + 3 \abs{Z_{2}} \right) + \abs{Z_{3}} + \abs{Y} \numberthis \label{second-ineq} \\
        & = \abs{Z_{1}} + \abs{Z_{2}} + \abs{Z_{3}} + \abs{Y} + \frac{1}{2} (\abs{Z_{1}} + \abs{Z_{2}}) \\
        & \leq n + \frac{n}{2} = \frac{3n}{2}
        \end{align*}
        where the inequality \eqref{first-ineq} follows from $d_{E_2}(v) \leq 1$ for $v \in Z_2$, and the inequality \eqref{second-ineq} follows from $\abs{Z_3} \leq \abs{Y}$ (second and third statements of Proposition \ref{prop}). We conclude that $\abs{E(H)} \leq \frac{3n}{2}$ contradicting the assumption that $H$ is a counterexample. \qed

    \section{Excluding critical configurations}
    \label{concluding-section}
    
    Here we prove Conjecture \ref{conj} for one particular critical configuration: an edge $e=(a,b,c)$ with $D(e) = \degvec{4,4,3}$ with link graph $G_6$ (see Figure \ref{fig-G6}). Note that although $D(e) < \degvec{4,4,4}$, $G_6$ cannot be obtained from any $G_i$ with $i \in [5]$ by deleting an edge, since those graphs (see Figure \ref{fig-link-graphs}) do not contain two vertex-disjoint four-cycles.
   
     \begin{figure}[H]
        \centering
        \includegraphics[width=0.6\linewidth]{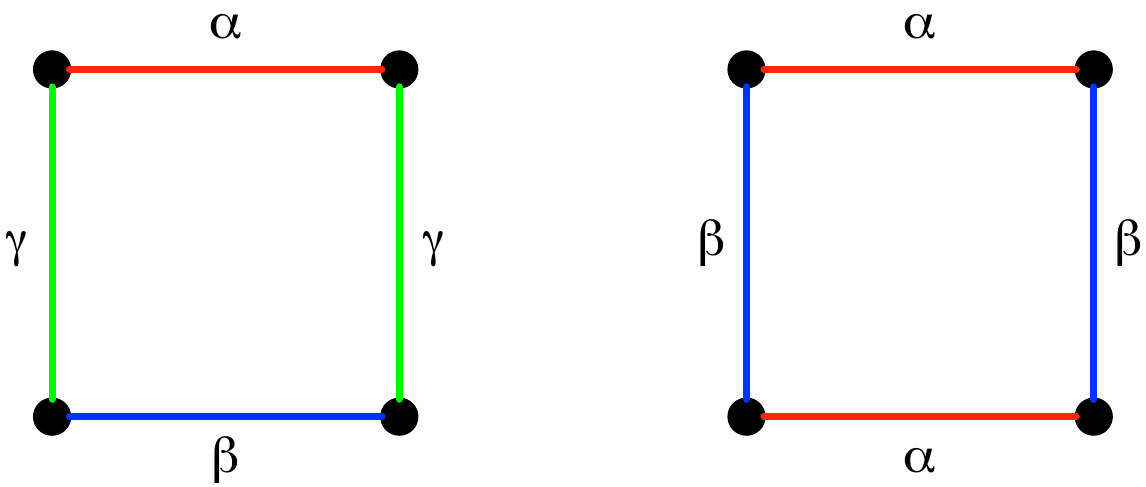}
        \caption{Link graph $G_6$}
        \label{fig-G6}
    \end{figure}
    
    \begin{proof}
        Let $H$ be a minimal counterexample with $n$ vertices containing an edge $e=(a,b,c)$ with $D(e) = \degvec{4,4,3}$ and with link graph $G_6$.
        
        Set $X = V(G_6) \cup \{a,b,c\}$. We prove that $X$ is incident to at most $\frac{3 \abs{X}}{2} = \frac{3 \times 11}{2} = 16.5$ edges of $H$. This will conclude the proof because after removing $X$ and its incident edges from $H$, we could have a smaller counterexample.
        
        In fact, we claim more: Apart from the extensions of the $8$ edges of $G_6$ and the edge $e$, $X$ can be incident only to those four possible edges of $H$ that contain a diagonal of the two four-cycles of $G_6$. Indeed, suppose that $f \in E(H)$ is not among these $14$ possibilities:
        
        \begin{enumerate}
            \item $a$ or $b$ is in $f$, without loss of generality, $a \in f$. Note that $f$ intersect $V(G_6)$ in at most one point and that point need to be one of the endpoints of the $\beta$-edge of the first component of $G_6$. Let $g$ be the edge containing $c$  and the $\gamma$-edge of the first component of $G_6$ which does not intersect $f$.  Let $h$ be an edge containing $b$ and a $\beta$-edge of the second component of $G_6$. Then $f,g,h$ are the jewels of a crown with base $e$, contradiction.
            
            \item $c$ is in $f$. Note that $f$ intersect $V(G_6)$ in at most one point and that point need to be on the second component of $G_6$. Let $g$ be the edge containing $b$  and the $\beta$-edge of the second component of $G_6$ which does not intersect $f$.  Let $h$ be the edge containing $a$ and the $\alpha$-edge of the first component of $G_6$. Then $f,g,h$ are the jewels of a crown with base $e$, contradiction.
            
            \item $f$ does not intersect $e$ and $x_1=V(G_6)\cap f$ is in the first component of $G_6$. We can select an $\alpha$-$\gamma$ path (or a $\beta$-$\gamma$ path) $x_1,x_2,x_3$ in the first component of $G_6$ and a $\beta$ (or an $\alpha$)  edge $(x_4,x_5)$ in the second component of $G_6$ such that this edge does not contain the possible intersection point of $f$ with the second component. Applying Lemma \ref{quintuple-lemma} with $Q=\{x_1,x_2,x_3,x_4,x_5\}$ leads to a contradiction.
            
            \item $f$ does not intersect $e$ and $x_1=V(G_6)\cap f$ is in the second component of $G_6$.  We can select an $\alpha$-$\beta$ path $x_1,x_2,x_3$ in the second component of $G_6$ and a $\gamma$-edge $(x_4,x_5)$ in the first component of $G_6$ such that this edge does not contain the possible intersection point of $f$ with the first component. Again, applying Lemma \ref{quintuple-lemma} with $Q=\{x_1,x_2,x_3,x_4,x_5\}$ leads to a contradiction.
        \end{enumerate}
    \end{proof}

\end{document}